\documentclass[11pt,a4paper,draft]{article}

\usepackage[ngerman,english]{babel}
\usepackage[utf8]{inputenc}

\usepackage[DIV=12]{typearea}

\usepackage[affil-it]{authblk}

\usepackage{xparse}
\usepackage{amsmath, amsfonts, amssymb, amsthm}
\usepackage{mathrsfs}
\usepackage{mathtools}
\usepackage{enumitem}

\usepackage[noadjust]{cite}

\usepackage[capitalise]{cleveref}

\crefname{equation}{}{}

\newcommand{\R}{\mathbb{R}}
\newcommand{\Z}{\mathbb{Z}}
\newcommand{\N}{\mathbb{N}}

\newcommand{\ee}{\mathrm{e}}

\DeclareDocumentCommand\dd{ o g d() }{
	\IfNoValueTF{#2}{
		\IfNoValueTF{#3}
			{\mathrm{d}\IfNoValueTF{#1}{}{^{#1}}}
			{\mathinner{\mathrm{d}\IfNoValueTF{#1}{}{^{#1}}\argopen(#3\argclose)}}
		}
		{\mathinner{\mathrm{d}\IfNoValueTF{#1}{}{^{#1}}#2} \IfNoValueTF{#3}{}{(#3)}}
	}

\newcommand{\dx}{\dd{x}}
\newcommand{\dy}{\dd{y}}

\newcommand{\del}{\partial}
\newcommand{\eps}{\varepsilon}
\newcommand{\im}{\mathrm{i}}

\newcommand{\cc}{\text{c.c.}} 

\DeclareMathOperator{\Res}{Res}
\DeclareMathOperator{\Id}{Id}
\DeclareMathOperator{\supp}{supp}

\DeclareMathOperator{\Lin}{\mathcal{L}}
\DeclareMathOperator{\Nlin}{\mathcal{N}}

\newcommand{\Ord}{\mathcal{O}}

\newcommand{\vcc}{\vcentcolon}

\DeclarePairedDelimiter\abs{\lvert}{\rvert}
\DeclarePairedDelimiter\norm{\Vert}{\rVert}

\newcommand{\B}{\mathcal{B}}

\newcommand{\Qr}{Q_{r}} 
\newcommand{\Kr}{K_{r}} 
\newcommand{\Qs}{Q_{s}} 
\newcommand{\Ks}{K_{s}} 
\newcommand{\Sw}{\mathscr{S}} 
\newcommand{\F}{\mathcal{F}} 

\newcommand{\ball}{I} 

\def\txtc{{\textnormal{c}}}
\def\txtd{{\textnormal{d}}}
\def\txte{{\textnormal{e}}}

\def\txti{{\textnormal{i}}}

\def\txtD{{\textnormal{D}}}

\newcommand{\cL}{{\mathcal L}}  
\newcommand{\cM}{{\mathcal M}}  
\newcommand{\cN}{{\mathcal N}}  
\newcommand{\cO}{{\mathcal O}}  

\theoremstyle{plain}
\newtheorem{theorem}{Theorem}[section]
\newtheorem{lemma}[theorem]{Lemma}
\newtheorem{proposition}[theorem]{Proposition}

\theoremstyle{definition}

\theoremstyle{remark}
\newtheorem{remark}[theorem]{Remark}

\numberwithin{equation}{section}

 \title{Validity of amplitude equations\\ for non-local non-linearities}

 \author{Christian Kühn\thanks{Technical University of Munich, Faculty of Mathematics, Research Unit 
``Multiscale and Stochastic Dynamics'', 85748 Garching b.~M\"unchen, Germany, \texttt{ckuehn@ma.tum.de}} }
 \author{Sebastian Throm\thanks{Technical University of Munich, Faculty of Mathematics, Research Unit 
``Multiscale and Stochastic Dynamics'', 85748 Garching b.~M\"unchen, Germany, \texttt{throm@ma.tum.de}} }
 \affil{}
 
 \date{}

\begin{document}

\maketitle

\begin{abstract}
Amplitude equations are used to describe the onset of instability in wide classes of
partial differential equations (PDEs). One goal of the field is to determine 
simple universal/generic PDEs, to which many other classes of equations can be 
reduced, at least on a sufficiently long approximating time scale. In this work, we study
the case, when the reaction terms are \emph{non-local}. In particular, we consider 
quadratic and cubic convolution-type non-linearities.
As a benchmark problem, we use the Swift-Hohenberg equation. The resulting amplitude equation
is a Ginzburg-Landau PDE, where the coefficients can be calculated from the kernels. 
Our proof relies on separating critical and non-critical modes in Fourier space in
combination with suitable kernel bounds.
\end{abstract}


\section{Introduction}\label{Sec:intro}

In this work we study the non-local Swift-Hohenberg (SH) equation
\begin{equation}
\label{eq:SH}
\partial_t u = -(1+\partial_{x}^2)^2u+pu+uQ\ast u+uK\ast u^2, 
\end{equation}
where $(x,t)\in\R \times[0,\infty)$, $p\in\R$ is a (small) parameter, 
$u=u(x,t)\in\R$, and $K,Q$ are given symmetric, finite measures; 
here $\ast$ denotes convolution in the spatial coordinate. The terms 
$uQ\ast u$ and $u K\ast u^2$ are the quadratic and cubic non-local non-linearities 
in~\eqref{eq:SH}. Before we discuss our main result, we provide a brief
overview of amplitude (or modulation) equations as well as recent 
results on non-local PDEs, which provide considerable motivation to
consider~\eqref{eq:SH}. The rigorous analysis of~\eqref{eq:SH} 
and our mathematical contribution starts in the next section.

In dynamical systems, one common approach to deal with local instabilities
is to derive a standard system, which represents the dynamics
of an entire class~\cite{GH}. Consider an ordinary differential equation (ODE)
\begin{equation*}
\frac{\txtd z}{\txtd t}=f(z;p),\qquad z=z(t)\in\R^d,~p\in\R,
\end{equation*}
with an equilibrium point $z_*$ undergoing a local bifurcation upon variation
of a parameter $p$, say at $p=0$. The standard method consists in first 
deriving a low-dimensional center manifold $\cM^\txtc$~\cite{Carr}. The manifold $\cM^\txtc$ is tangent to the center eigenspace of $\txtD f(z_*;0)$. On 
$\cM^\txtc$, the dynamics is low-dimensional and can be brought into 
a \emph{normal form} by first Taylor-expanding and then using coordinate 
transformations to eliminate as many polynomial terms up to a given 
order~\cite{Kuznetsov,GH}. This procedure 
yields several generic classes of low-dimensional ODEs, which can then
be analysed.

A similar strategy is available for many PDEs~\cite{Hoyle,CrossHohenberg,
CrossGreenside}. A typical class is
\begin{equation}
\label{eq:baseRDE}
\partial_t u = \cL u + F(u;p),\qquad u=u(x,t),~(x,t)\in\Omega\times[0,T),~
\Omega\subseteq\R^d,  
\end{equation}     
where $\cL$ is a linear differential operator and $F(u;p)=F$ is the 
non-linearity. Suppose $u_*$ is a steady state
of~\eqref{eq:baseRDE} for all $p$. If the spectrum $\sigma(\cL+
\txtD F(u_*;p))$ is contained in the left-half of the complex plane, 
then $u_*$ is locally linearly stable~\cite{Henry}. Upon parameter 
variation of $p$, say without loss of generality at $p=0$, a bifurcation occurs, when 
$\sigma(\cL+\txtD F(u_*;0))\cap \txti \R\neq \emptyset$ and suitable
genericity conditions hold, i.e., transversal crossing of the spectrum 
and non-degeneracy of the non-linearity~\cite{Hoyle,CrossHohenberg}. 
As for the ODE, we may ask, whether there is a \emph{simple generic} 
normal form, now called \emph{amplitude} or \emph{modulation} equation, 
to describe the formation of non-trivial patterns near $p=0$. For a 
bounded domain $\Omega$ and suitable $\cL$, 
one may often use standard centre manifold reduction for point spectrum 
crossing $\txti\R$ at $p=0$~\cite{BatesJones,VanderbauwhedeIooss}. 
However, for cases involving unbounded domains, one usually faces 
essential spectrum crossing $\txti\R$, which presents substantial
challenges as one expects the amplitude equation to be a PDE, not an ODE,
in this context~\cite{CrossHohenberg}.

The development of the field of amplitude equations has a long history
and a benchmark problem is to consider the \emph{local} Swift-Hohenberg
equation~\cite{BurkeKnobloch,LegaMoloneyNewell,LloydSandstedeAvitabileChampneys,
Schneider5} 
\begin{equation}
\label{eq:SHloc}
\partial_t u = -(1+\Delta)^2u+pu+\cN(u), \qquad x\in\R^d,
\end{equation}
where $\cN(u)$ is a given non-linearity, frequently taken as a quadratic-cubic 
polynomial. The spectrum of the linearised 
operator has two quadratic tangencies with $\txti \R$ for $p=0$. To derive 
an amplitude equation \emph{formally}, one possibility is to use the method of
multiple scales~\cite{KevorkianCole,BenderOrszag,Hoyle,KuehnBook} in combination 
with the ansatz
\begin{equation}
\label{eq:ansatzA}
u(x,t)\approx \psi_A(x,t)= \varepsilon(A(X,T)\txte^{\txti k\cdot x}+\overline{A}(X,T)
\txte^{-\txti k\cdot x}),
\end{equation} 
where $X$ are the new scaled variables with $X_{i}=x_{i}\varepsilon^{a_i}$
for some exponents $a_i>0$, $T=t\varepsilon^{b}$ is a scaled time for some 
$b>0$, $k\in \Z^d$ is a suitably chosen wave vector, and $A$ is a slowly modulated
amplitude governing the envelope of the fast Fourier modes. One 
re-writes~\eqref{eq:SHloc} using the doubled number of variables 
$(x,t;X,T)$ via the chain rule, inserts an asymptotic series
\begin{equation*}
u=u_0(x,t;X,T)+\varepsilon u_1(x,t;X,T) + \varepsilon^2 u_2(x,t;X,T)+\cdots
\end{equation*}
into the resulting PDE, and then uses~\eqref{eq:ansatzA} to derive a PDE
for $A(X,T)$. For example, $d=1$, $\cN(u)=-u^3$, and $k=1$ yield the (real) 
Ginzburg-Landau equation~\cite{AransonKramer,Mielke1,LevermoreOliver}
\begin{equation}
\label{eq:RGL}
\partial_T A=4\partial_{X}^2+\hat{p}A-3A|A|^2,\qquad p/\varepsilon^2=:\hat{p}\in\R.
\end{equation}  
The next step is to prove \emph{rigorous validity} of the approximation, 
which has been discussed in many publications for \emph{local} 
PDEs; see e.g., \cite{KirrmannSchneiderMielke,vanHarten,Schneider1,
Schneider2,Schneider3}.
The typical structure of the approximation results is the following: 
Assume the amplitude equation has a solution $A$ of a certain regularity
over a time scale $T\in [0,T_*]$ and $\|\psi_A(\cdot,0)-u(\cdot,0)\|\lesssim 
\varepsilon^{\alpha_0}$ for all $x$ in the spatial domain. Then one proves 
\begin{equation}
\|\psi_A(\cdot,t)-u(\cdot,t)\|\lesssim \varepsilon^{\alpha},\quad 
\text{for all $t\in\left[0,\frac{T_*}{\varepsilon^{\beta}}\right]$,} 
\end{equation} 
where various choices of the (space-variable) norm $\|\cdot\|$ 
can be considered.
For example, in the case \eqref{eq:SHloc}--\eqref{eq:RGL} with
$\alpha_0=2$, $\alpha=2$, and $\beta=2$ one may prove a uniform
pointwise $\cO(\varepsilon^2)$-approximation over a long time scale of order 
$\cO(T_*/\varepsilon^2)$~\cite{KirrmannSchneiderMielke}.

There have been several recent works, using a multiple scales
approach to \emph{formally} derive amplitude equations also in the
case when the non-linearity is \emph{non-local}. Morgan and 
Dawes~\cite{MorganDawes} studied a Swift-Hohenberg 
equation~\eqref{eq:SHloc} for $d=1$ with \emph{non-local non-linearity} 
\begin{equation}
\label{eq:MD}
\cN(u)=c_2u^2-u^3-c_3u\int_{\R} K(\cdot-y)u(y,t)^2~\txtd y,
\end{equation}
where $c_2,c_3\in\R$ are parameters.
They provided the formal derivation of the amplitude equation in the
case~\eqref{eq:MD}, calculated the coefficients in the Ginzburg-Landau 
equation for two classes of the kernel $K$ explicitly, and provided 
numerical bifurcation studies of the amplitude equation.
Hence, their work provides \emph{immediate motivation} to investigate the 
\emph{rigorous validity} of amplitude equations for our non-local Swift-Hohenberg
equation~\eqref{eq:SH}. Indeed,~\eqref{eq:MD} is just a special case of
the non-linearity in~\eqref{eq:SH} as we allow $\delta$-measures to 
appear in the kernels $Q,K$. Faye and Holzer~\cite{FayeHolzer} have studied 
the non-local Fisher-Kolmogorov-Petrovskii-Piscounov (FKPP) equation~\cite{Gourley}
\begin{equation}
\label{eq:FH}
\partial_t u = \partial_{x}^2u+c_2u\left(1-\int_\R Q(\cdot-y)u(y,t)~\txtd y\right),
\end{equation}
which has raised considerably interest recently in the literature; see
e.g.~\cite{BerestyckiNadinPerthameRyzhik,BerestyckiJinSilvestre,GenieysVolpertAuger,AchleitnerKuehn,
AlfaroCoville}. Faye and Holzer are interested in modulated travelling fronts 
bifurcating from the monotone FKPP invasion wave upon variation of $c_2$. Part 
of their work~\cite[Sec.3]{FayeHolzer}, contains a multiple scales ansatz to
derive amplitude equations for the modulated fronts, which again yields
a Ginzburg-Landau equation with coefficients that can be calculated from
the kernel $Q$. As in the case of~\eqref{eq:MD}, also~\eqref{eq:FH}
provides strong motivation to investigate non-local non-linearities and
related amplitude equations in more detail.

The model problem~\eqref{eq:SH} can also be motivated
more abstractly. It contributes to the \emph{general interest} to obtain a
better understanding of non-local non-linearities. Examples 
include neural field equations~\cite{Bressloff,Coombes}, phase-field 
models~\cite{BatesFifeRenWang,Chen1}, non-local singular perturbation 
problems~\cite{Bose,Fife4}, various types of reaction-diffusion 
PDEs~\cite{Souplet,RubinsteinSternberg,SiebertAlonsoBaerSchoell}, 
non-local Schr\"odinger equations~\cite{AblowitzMusslimani,Ackermann}, 
non-local models in vegetation pattern formation~\cite{Meron1,Siero} and
vast classes of PDEs with 
constraints, e.g., elliptic-parabolic systems with elliptic part 
solvable in integral form. For all these scenarios, rigorous results on
amplitude equations are going to be relevant.

Our main result for~\eqref{eq:SH} can informally be stated as follows: Recall
$Q,K$ are finite measures, which are symmetric, so that 
they obey the same symmetry of the spectrum of the linearised Swift-Hohenberg 
equation. Consider the \emph{local} Ginzburg-Landau equation for an amplitude 
$A$, where the coefficients in this equation can be calculated from the Fourier 
transforms of $Q,K$. Suppose $A(X,T)$ is a sufficiently regular solution 
for $T\in [0,T_*]$ and $\|\psi_A(\cdot,0)-u(\cdot,0)\|_{C^4}\lesssim \varepsilon^2$, then
\begin{equation}
\|\psi_A(\cdot,t)-u(\cdot,t)\|_{C^4}\lesssim \varepsilon^{2},\quad 
\text{for all $t\in\left[0,\frac{T_*}{\varepsilon^{2}}\right]$,} 
\end{equation}  
with $\psi_A(x,t)=\varepsilon(A(\varepsilon x,\varepsilon^2t)\txte^{\txti x}+
\overline{A}(\varepsilon x,\varepsilon^2t)\txte^{-\txti x})$; see also 
Theorem~\ref{Thm:main}.

In some sense, our result is the natural analogue to the classical local result, 
and we include the local result as a special case in our approach. The key
proof strategy is to generalise techniques from~\cite{Schneider4,MielkeSchneider1} 
to the non-local case using suitable a priori bounds. Although these 
bounds do not yield the exact cancellation property initially developed
in~\cite{KirrmannSchneiderMielke} via an improved higher-order approximation, 
the kernel bounds do still yield the correct error order, i.e., only produce 
terms of order $\cO(\varepsilon^3)$ in the final result. Our method is designed to 
be general enough to handle larger classes of PDEs,
not just~\eqref{eq:SH}, as we only use the spectral information from the 
linear part, and the non-linearity contains the first two important forms of 
quadratic and cubic terms. However, the Swift-Hohenberg model problem already 
shows very clearly the key steps required in the analysis. In summary, our
results provide a step towards a more general theory of amplitude equations
for non-local PDEs.

\section{Assumptions and main result}\label{Sec:results}

We now specify the assumptions used throughout this work and we precisely state the main result we are going to show. Therefore, we recall~\eqref{eq:SH} and note that the considerations in Section~\ref{Sec:intro} suggest the scaling $p=\eps^2$ with a small parameter $\eps>0$. Thus, we are led to study the equation
\begin{equation}\label{eq:SW}
 \del_{t}u=-(1+\del_{x}^{2})^{2}u+\eps^{2} u-uQ\ast u-uK\ast u^2 \quad \text{on }\R.
\end{equation}
The precise assumptions on the convolution kernels $Q$ and $K$ will be given below (see Section~\ref{Sec:assumptions}).

\subsection{Assumptions on $Q$ and $K$}\label{Sec:assumptions}

In the remainder of this work, the convolution kernels $Q$ and $K$ are assumed to be finite, symmetric measures on $\R$, i.e.\@ $Q,K\in\mathcal{M}^{\text{fin}}(\R)$ and symmetric, such that it holds 
\begin{equation}\label{eq:ker:moment:bd}
 \int_{\R}\abs{x}\abs{Q}(\dx)<\infty\quad \text{and}\quad \int_{\R}\abs{x}\abs{K}(\dx)<\infty.
\end{equation}

\begin{remark}
 Note that throughout this work, we will use the notation $Q(\dx)=Q(x)\dx$ and $K(\dx)=K(x)\dx$ for simplicity although $Q$ and $K$ might not have a Lebesgue density. Thus all integrals occurring have to be interpreted accordingly. In particular, we write by abuse of notation $\norm{fQ}_{L^{1}}$ for the expression $\int_{\R}\abs{f}\abs{Q}(\dx)$ and equivalently also for $K$.
\end{remark}
Since our analysis relies crucially on the use of the Fourier transform, we have to restrict moreover at certain places to measures $Q$ and $K$ which can be decomposed as
\begin{equation}\label{eq:QK:decomposition}
 \begin{gathered}
   Q=\Qr+\Qs\quad \text{and}\quad K=\Kr+\Ks\quad \text{with } \Qr,\Kr\in\Sw\\
   \text{and }\Qs,\Ks\in\mathcal{M}^{\text{fin}}(\R)\text{ compactly supported}.
 \end{gathered}
\end{equation}
Here, $\Sw$ denotes the Schwartz space of smooth and rapidly decaying functions, while the dual space of tempered distributions will be denoted by $\Sw'$ throughout this work. We emphasise that the usual embedding $\mathcal{M}^{\text{fin}}(\R)\subset \Sw'$ yields that we may consider $\Qs$ and $\Ks$ also as elements in $\Sw'$.

\begin{remark}\label{Rem:assumptions:kernels}
 We note that~\eqref{eq:QK:decomposition} in particular implies~\eqref{eq:ker:moment:bd} while~\eqref{eq:QK:decomposition} also allows for purely 'local' non-linearities if we choose $Q=q\delta(\cdot)$ and $K=k\delta(\cdot)$.
\end{remark}

 To simplify the notation at several places, we define
 \begin{equation*}
  \begin{gathered}
     \Lin(v)\vcc=-(1+\del_{x}^{2})^{2}v+\eps^{2}v,\quad \Nlin_{Q}(v)\vcc=-u Q\ast u, \quad \Nlin_{K}(v)=-u K\ast u^2\\
     \quad \text{and}\quad \Nlin(v)=\Nlin_{Q}(v)+\Nlin_{K}(v)
  \end{gathered}
 \end{equation*}
 such that~\eqref{eq:SW} can be written as
 \begin{equation*}
  \del_{t}u=\Lin(u)+\Nlin(u).
 \end{equation*}

Moreover, we have the following continuity property for the convolution operators induced by $Q$ and $K$.

\begin{lemma}\label{Lem:continuity:convolution}
 For each $n\in\Z$ and $k\in \N_{0}$ there exists a constant $C>0$ such that it holds
 \begin{align}
  \norm{B_{1}(Q\ee^{\im n\cdot})\ast B_{2}}_{C^{k}}&\leq C\norm{B_{1}}_{C^{k}}\norm{B_{2}}_{C^{k}}\label{eq:continuity:1}
  \shortintertext{and}
  \norm{B_{1}(K\ee^{\im n\cdot})\ast (B_{2}B_{3})}_{C^{k}}&\leq C\norm{B_{1}}_{C^{k}}\norm{B_{2}}_{C^{k}}\norm{B_{3}}_{C^{k}}\label{eq:continuity:2}
 \end{align}
 for all $B_{\ell}\in C^{k}_{\text{b}}(\R)$ with $\ell\in\{1,2,3\}$.
\end{lemma}

\begin{proof}
 Due to the assumptions of $Q$ it holds 
 \begin{equation}\label{eq:continuity:proof:1}
  \abs*{B_{1}(Q\ee^{\im n\cdot})\ast B_{2}}\leq \norm{B_{1}}_{C^{0}}\norm{B_{2}}_{C^{0}}\norm{Q\ee^{\im n\cdot}}_{L^{1}}\leq C\norm{B_{1}}_{C^{0}}\norm{B_{2}}_{C^{0}}.
 \end{equation}
 This proves~\eqref{eq:continuity:1} for $k=0$, while the case of general $k\in \N$ then follows immediately from~\eqref{eq:continuity:proof:1} together with Leibniz' rule. The proof of~\eqref{eq:continuity:2} is analogously.
\end{proof}

Finally, we introduce some notation, i.e.\@ the assumption~\eqref{eq:ker:moment:bd} allows to define the constants
\begin{equation}\label{eq:correctors}
 q_{n}\vcc=\int_{\R}\ee^{\im n x}Q(x)\dx\quad \text{and}\quad k_{n}\vcc=\int_{\R}\ee^{\im n x}K(x)\dx\quad \text{for all }n\in\Z.
\end{equation}
Moreover, we note that due to the symmetry of $Q$ and $K$ it also holds
\begin{equation*}
 k_{-n}=k_{n}\quad \text{and}\quad q_{-n}=q_{n}\quad \text{for all }n\in\Z
\end{equation*}
and thus in particular $k_{n},q_{n}\in\R$ for all $n\in\Z$.

\subsection{Main result}

As explained in Section~\ref{Sec:intro}, one expects that a solutions $u$ of~\eqref{eq:SW} can be approximated by a function of the form
\begin{equation}\label{eq:approximation}
 \psi(x,t)=\eps\bigl(A(\eps x,\eps^{2} t)\ee^{\im x}+\bar{A}(\eps x,\eps^{2}t)\ee^{-\im x}\bigr)
\end{equation}
provided that the initial condition $u_{0}=u(\cdot,0)$ is sufficiently close to $\psi(\cdot,0)$ and $A$ is a solution to the Ginzburg-Landau equation. Precisely, under our assumptions on $Q$ and $K$ formal calculations suggest to take $A$ as solution of 
\begin{equation}\label{eq:GL}
 \del_{T}A(X,T)=(1+4\del_{X}^{2})A(X,T)-\Bigl(2k_{0}+k_{2}-\frac{q_{1}q_{2}}{9}-\frac{q_{1}^{2}}{9}-2q_{0}q_{1}-2q_{1}^{2}\Bigr)\abs*{A(X,T)}^{2}A(X,T).
\end{equation}

The following proposition guarantees the existence of a solution to~\eqref{eq:GL} at least locally in time.

\begin{proposition}\label{Prop:GL:solvability}
 Assume that $A_{0}(\cdot)\in C_{\text{b}}^{4}(\R)$. Then there exists $T_{*}>0$ such that there exists a unique solution $A=A(X,T)\in C\bigl([0,T_{*}],C^{4}_{\text{b}}\bigr)$ of~\eqref{eq:GL} with $A(\cdot,0)=A_{0}$.
\end{proposition}

\begin{proof}
 This statement follows easily by an application of the contraction mapping theorem.
\end{proof}

We can now state the main result that we will show in this work.

\begin{theorem}\label{Thm:main}
 Let $A\in C\bigl([0,T_{*}],C^{4}_{\text{b}}\bigr)$ be a solution of~\eqref{eq:GL}. Then for each $d>0$ there exist constants $\eps_{*},C>0$ such that for all $\eps\in(0,\eps_{*})$ the following statement holds. If $\norm*{u_{0}-\psi(\cdot,0)}_{C^{4}}\leq d\eps^{2}$ then there exists a unique solution $u$ of~\eqref{eq:SW} on the time interval $[0,T_{*}/\eps^{2}]$ with $u(\cdot,0)=u_{0}$ and moreover we have the estimate
 \begin{equation*}
  \norm{u(\cdot,t)-\psi(\cdot,t)}_{C^{4}}\leq C\eps^{2}\quad \text{for all }(x,t)\in\R \times [0,T_{*}/\eps^{2}].
 \end{equation*}
\end{theorem}

\subsection{Notation and outline}

In order to prove Theorem~\ref{Thm:main}, we will follow the same approach as in~\cite{Schneider4} and thus, instead of showing that $\psi$ is a good approximation for solutions of~\eqref{eq:SW}, we will consider the intermediate approximation 
\begin{multline}\label{eq:approx:refined}
 \phi(x,t)=\eps\bigl((E_{0}A)(X,T)\ee^{\im x}+(E_{0}\bar{A})(X,T)\ee^{-\im x}\bigr)\\*
 +\eps^{2}\bigl((E_{0}A_{2})(X,T)\ee^{2\im x}+(E_{0}\bar{A}_{2})(X,T)\ee^{-2\im x}+(E_{0}A_{0})(X,T)\bigr)
\end{multline}
 with $X=\eps x$ and $T=\eps^{2}t$. The operator $E_{0}$ acts as a cut-off function in Fourier variables to select modes which are sufficiently close to zero. The precise definition of $E_{0}$ is given in~\eqref{eq:mode:filter}. The coefficients $A$, $A_{0}$ and $A_{2}$ are chosen such that $A$ is a solution of~\eqref{eq:GL} while $A_{2}$ and $A_{0}$ are given by
 \begin{equation}\label{eq:A2:A0}
  A_{0}=-2q_{1}\abs*{A}^{2}\quad \text{and}\quad A_{2}=-\frac{q_{1}}{9}A^{2}.
 \end{equation}
 One key ingredient in the proof of Theorem~\ref{Thm:main} is to consider the critical Fourier modes $\ee^{\pm \im x}$ separately from the uncritical ones. Therefore, one defines
  \begin{equation}\label{eq:phi:c:phi:s}
  \phi_{c}=(E_{0}A)\ee^{\im x}+(E_{0}\bar{A})\ee^{-\im x}\quad  \text{and}\quad \phi_{s}=(E_{0}A_{2})\ee^{2\im x}+(E_{0}\bar{A}_{2})\ee^{-2\im x}+(E_{0}A_{0})
 \end{equation}
 such that $\phi=\eps \phi_{c}+\eps^{2}\phi_{s}$. We then have the following lemma which states that $\phi_{s}$ is uniformly bounded and $\phi$ is uniformly close to $\psi$ up to an error of $\Ord(\eps^{2})$.
  
 \begin{lemma}\label{Lem:phi:estimates}
 Let $A\in C([0,T_{*}],C_{\text{b}}^{4}(\R))$ be a solution of~\eqref{eq:GL} and $\phi_{c}$ and $\phi_{s}$ be given by~\eqref{eq:phi:c:phi:s} together with~\eqref{eq:A2:A0}. Then there exists a constant $C>0$ such that it holds
 \begin{equation*}
  \sup_{t\in[0,T_{*}/\eps^2]}\norm{\phi_{s}}_{C^{4}}\leq C\quad \text{and}\quad \sup_{t\in[0,T_{*}/\eps^{2}]}\norm{\phi-\psi}_{C^{4}}\leq C\eps^{2},
 \end{equation*}
where $\phi=\eps\phi_{c}+\eps^{2}\phi_{s}$ and $\psi$ is defined in~\eqref{eq:approximation}.
\end{lemma}

\begin{proof}
 The bound on $\phi_{s}$ is an immediate consequence of the assumptions on $A$, the definition of $A_{0}$ and $A_{2}$ in~\eqref{eq:A2:A0} and Leibniz' rule, while we also note that the operator $E_{0}$ commutes with $\del_{x}$.
 
 To verify the second estimate of the lemma, we note that
 \begin{equation*}
  \begin{split}
   (\phi-\psi)(x)&=\eps\Bigl(\bigl(E_{0}A(\eps \cdot)\bigr)(x)\ee^{\im x}+\bigl(E_{0}\bar{A}(\eps\cdot)\bigr)(x)\ee^{-\im x}-A(\eps x)\ee^{\im x}-\bar{A}\ee^{-\im x}\Bigr)+\eps^{2}\phi_{s}\\
   &=\eps^{2}\phi_{s}+\eps\Bigl(\bigl(E_{0}^{c}A(\eps \cdot)\bigr)(x)\ee^{\im x}+\bigl(E_{0}^{c}\bar{A}(\eps\cdot)\bigr)\ee^{-\im x}\Bigr),
  \end{split}
 \end{equation*}
 where $E_{0}^{c}$ is defined in~\eqref{eq:mode:filter}. Thus, combining Lemma~\ref{Lem:smoothing} below with Leibniz' rule as well as the first estimate of the lemma, the claim easily follows.
\end{proof}
 
 The main strategy to prove Theorem~\ref{Thm:main} is now as follows. First, we note that Lemma~\ref{Lem:phi:estimates} yields that $\psi$ can be approximated by $\phi$ on the time interval $[0,T_{*}/\eps^2]$ up to an error of $\Ord(\eps^{2})$. As a consequence, it is enough to prove Theorem~\ref{Thm:main} with $\psi$ replaced by $\phi$. The general approach for this will be to consider the approximation error $R=u-\phi$ and to show that this quantity remains of $\Ord(\eps^{2})$ on $[0,T_{*}/\eps^{2}]$. To this end, we will derive an evolution equation for $R$ and show that there exists a unique solution which is small on $[0,T_{*}/\eps^{2}]$. Since $u$ is on the other hand uniquely determined on a small time interval, this then yields that $u$ also exists on $[0,T_{*}/\eps^{2}]$ by a standard continuation argument. One crucial part within this approach consists in obtaining suitable estimates for the residuum of $\phi$ which is defined as
 \begin{equation}\label{eq:Res}
  \Res(\phi)\vcc=-\del_{t}\phi+\Lin(\phi)+\Nlin(\phi).
 \end{equation}
 The study of this expression will be contained in Section~\ref{Sec:residuum}. Moreover, in Section~\ref{Sec:equation:error}, we will derive the equation which has to be satisfied by $R$, while we already note, that in order to obtain that $R$ stays of $\Ord(\eps^2)$ on $[0,T_{*}/\eps^{2}]$, it will be necessary to consider the critical and uncritical modes separately. Based on these preparations, we will then provide the proof of Theorem~\ref{Thm:main} in Section~\ref{Sec:Proof:main}. Moreover, in Section~\ref{Sec:Fourier}, we recall several technical definitions and properties from~\cite{Schneider4} which will be used frequently. 
 
 \subsection{Main difference to \emph{local} non-linearity}
 
 To conclude this section, we will finally point out one main difference between the proof of Theorem~\ref{Thm:main} and the corresponding result for \emph{local} non-linearities, i.e.\@ the equation
 \begin{equation*}
  \del_{t}u=-(1+\del_{x}^{2})^{2}u+\eps^{2} u-qu^2-ku^3
 \end{equation*}
as for example considered in~\cite{MielkeSchneider1}. However, as mentioned in Remark~\ref{Rem:assumptions:kernels}, this equation is still contained as special case in our Theorem~\ref{Thm:main}.

As explained above, we follow the same main approach as in~\cite{MielkeSchneider1,Schneider4} by computing and estimating $\Res(\phi)$ in order to show that the approximation error $R=u-\phi$ remains small. However, in the case where the non-linearity is given as $\cN(u)$ with a polynomial $\cN$, the choice of $A$ together with~\eqref{eq:A2:A0} yields that in $\Res(\phi)$ several expressions of lower order in $\eps$ exactly cancel. In contrast to this, when we consider the more general \emph{non-local} non-linearities as in~\eqref{eq:SW} this is no longer the case. To circumvent this problem, we have to use the following result which states that although the lower order expressions do not cancel, we can still gain at least one order in $\eps$.

\begin{lemma}\label{Lem:approx}
 For each $n\in\Z$ there exists a constant $C>0$ such that it holds
 \begin{align*}
  \norm{B_{1}(\eps\cdot)(Q\ee^{n\im\cdot})\ast B_{2}(\eps \cdot)-q_{n}(B_{1}B_{2})(\eps\cdot)}_{C^{1}}&\leq C\eps \norm{B_{1}}_{C^{1}}\norm{B_{2}}_{C^{1}} \\
  \norm{B_{1}(\eps\cdot)(K\ee^{n\im\cdot})\ast(B_{2}B_{3})(\eps\cdot)-k_{n}(B_1B_2B_3)(\eps\cdot)}_{C^{1}}&\leq C\eps \norm{B_1}_{C^{1}}\norm{B_2}_{C^{1}}\norm{B_3}_{C^{1}}
 \end{align*}
 for all $B_{\ell}\in C_{\text{b}}^{1}(\R)$ with $\ell\in\{1,2,3\}$.
\end{lemma}

\begin{proof}
 Since $\del_{x}B_{\ell}(\eps x)=\eps \del_{X}B_{\ell}(\eps x)$ for $\ell=1,2,3$ and $B_{\ell}\in C_{\text{b}}^{1}(\R)$ it suffices to prove that 
  \begin{align}
  \norm{B_{1}(\eps\cdot)(Q\ee^{n\im\cdot})\ast B_{2}(\eps \cdot)-q_{n}(B_{1}B_{2})(\eps\cdot)}_{C^{0}}&\leq C\eps \norm{B_{1}}_{C^{1}}\norm{B_{2}}_{C^{1}} \label{eq:approx:quadratic}\\
  \norm{B_{1}(\eps\cdot)(K\ee^{n\im\cdot})\ast(B_{2}B_{3})(\eps\cdot)-k_{n}(B_1B_2B_3)(\eps\cdot)}_{C^{0}}&\leq C\eps \norm{B_1}_{C^{1}}\norm{B_2}_{C^{1}}\norm{B_3}_{C^{1}} \label{eq:approx:cubic}.
 \end{align}
 We first consider~\eqref{eq:approx:quadratic} and notice that 
\begin{multline*}
 \abs*{B_{1}(\eps x)\bigl((Q\ee^{n\im \cdot})\ast B_{2}(\eps \cdot)\bigr)(x)-q_{n}(B_{1}B_{2})(\eps x)}\\*
 =\abs*{B_{1}(\eps x)\int_{\R}Q(y)\ee^{\im n y}\bigl(B_{2}(\eps(x-y))-B_{2}(\eps x)\bigr)\dy}\\*
 \leq \eps \norm{B_{1}}_{C^{0}}\norm{B_{2}}_{C^{1}}\int_{\R}\abs{yQ(y)}\dy\leq C\eps  \norm{B_{1}}_{C^{0}}\norm{B_{2}}_{C^{1}}.
\end{multline*}
Here we also used that $\abs*{B_{2}(\eps(x-y))-B_{2}(\eps x)}\leq \eps \norm{B_{2}}_{C^{1}}\abs{y}$ for all $x,y\in\R$. Thus, \eqref{eq:approx:quadratic} immediately follows. 

To prove~\eqref{eq:approx:cubic} we can argue analogously since we have the relation
\begin{multline*}
 \abs*{B_{1}(\eps x)\bigl((K\ee^{n\im\cdot})\ast(B_{2}B_{3})(\eps\cdot)\bigr)(x)-k_{n}(B_1B_2B_3)(\eps x)}\\*
 =\biggl|B_{1}(\eps x)\int_{\R}K(y)\ee^{\im n y}\bigl(B_{2}(\eps(x-y))-B_{2}(\eps x)\bigr)B_{3}(\eps(x-y))\dy\\*
 +B_{1}(\eps x)\int_{\R}K(y)\ee^{n\im y}B_{2}(\eps x)\bigl(B_{3}(\eps(x-y))-B_{3}(\eps x)\bigr)\dy\biggr|.
\end{multline*}
From this, the estimate~\eqref{eq:approx:cubic} follows in the same way as for the quadratic term.
\end{proof}

\section{Technical preparation}\label{Sec:Fourier}

Our strategy to prove Theorem~\ref{Thm:main} follows closely that one in~\cite{Schneider4}, where the equation
\begin{equation*}
 \del_{t}u=-(1+\del_{x}^{2})^{2}u+\eps^{2} u+u\del_{x} u
\end{equation*}
has been considered and we thus recall in this section several technical fundamentals. Moreover, we provide the necessary adaptations and extensions that we need for the situation that we consider in this work. More precisely, we will work in the space $C^{4}_{\text{b}}(\R)$ of four times differentiable functions with globally bounded derivatives. As already indicated before, one key ingredient is to consider the critical Fourier modes $\ee^{\pm \im x}$ separately from the uncritical ones which will be achieved by suitable multiplication operators in Fourier space the so-called mode filters. This approach makes it necessary, to work with the Fourier transform which is not directly defined on the space $C^{4}_{\text{b}}(\R)$. However, as also pointed out in~\cite{Schneider4} we can embed $C^{4}_{\text{b}}(\R)$ into $\Sw'$, where the Fourier transform is defined in the usual way by duality. 

We recall now the definition of the mode filters as given in~\cite{Schneider4} and for this, we will denote by $\ball_{r}(x)$ the open interval of radius $r$ centred around $x$, i.e.\@ $\ball_{r}(x)=(x-r,x+r)$. One then fixes non-negative and even functions $\chi_{c},\chi_{0}\in C_{\text{c}}^{\infty}(\R)$ which satisfy
\begin{equation*}
 \chi_{c}(k)=\begin{cases}
              1 & \text{if } k\in \ball_{1/8}(-1)\cup \ball_{1/8}(1)\\
              0 & \text{if } k\in\R\setminus\bigl(\ball_{1/4}(-1)\cup \ball_{1/4}(1)\bigr)
             \end{cases}
             \quad\text{and}\quad 
 \chi_{0}(k)=\begin{cases}
              1 & \text{if } k\in \ball_{1/8}(0)\\
              0 & \text{if } k\in\R\setminus \ball_{1/4}(0).
             \end{cases}
\end{equation*}
For these functions we additionally define $G_{c}$ and $G_{0}$ as the inverse Fourier transforms, i.e.
\begin{equation*}
 G_{c}(x)=\frac{1}{2\pi}\int_{\R}\ee^{\im k x}\chi_{c}(k)\dd{k}\quad \text{and}\quad G_{0}(x)=\frac{1}{2\pi}\int_{\R}\ee^{\im k x}\chi_{0}(k)\dd{k}.
\end{equation*}
The mode filters $E_{c}$, $E_{0}$, $E_{0}^{c}$ and $E_{s}$ are then defined as
\begin{equation}\label{eq:mode:filter}
 E_{c} v\vcc=G_{c}\ast v, \quad E_{0}v\vcc=G_{0}\ast v, \quad  E_{0}^{c}\vcc=E_{0}-\Id \quad \text{and}\quad E_{s}\vcc=\Id-E_{c}.
\end{equation}

\begin{remark}
 If we denote by $\F$ the Fourier transform, it is well-known that for $v\in\Sw$ it holds $\F(E_{c} v)=\chi_{c}\F(v)$ as well as $\F(E_{0} v)=\chi_{0}\F(v)$. Moreover, since $\chi_{c}$ and $\chi_{0}$ have compact support it holds in particular $\chi_{c},\chi_{0}\in\Sw$ such that~\eqref{eq:mode:filter} makes even sense for $v\in\Sw'$ since the convolution between tempered distributions and Schwartz functions is well-defined. 
\end{remark}

\begin{remark}
 We additionally remark that $E_{s}$ and $E_{0}^{c}$ can also be represented as convolution operators, with kernels $G_{s}=\delta-G_{c}$ as well as $G_{0}^{c}=G_{0}-\delta$. The corresponding Fourier transforms are given by $\chi_{s}=1-\chi_{c}$ and $\chi_{0}^{c}=\chi_{0}-1$. Note that in this case $G_{0}^{c}$ and $G_{s}$ are only measures due the fact that the Fourier transforms have unbounded support.
\end{remark}

For technical reasons it is also necessary to introduce further operators $E_{c}^{h}$ and $E_{s}^{h}$ which satisfy $E_{c}^{h}E_{c}=E_{c}$ and $E_{s}^{h}E_{s}=E_{s}$ and which are defined via $C^{\infty}$-functions $\chi_{c}^{h}$ and $\chi_{s}^{h}$. More precisely, $\chi_{c}^{h}\in C_{\text{c}}^{\infty}(\R)$ is chosen such that it vanishes outside $\ball_{3/8}(-1)\cup \ball_{3/8}(1)$ while $\chi_{s}^{h}$ vanishes in $\ball_{1/16}(-1)\cup \ball_{1/16}(1)$.

With these definitions, we can cite three results on the mode filters which are contained in~\cite{Schneider4} as Lemmas~3--5.

\begin{lemma}\label{Lem:smoothing}
 The operators $E_{c}$ and $E_{0}$ are linear and continuous mappings from $C^{0}$ to $C^{m}$. For every $m\geq 0$ there exists $C_{m}>0$ with $\norm{E_{0}u}_{C^{m}}+\norm{E_{c}u}_{C^{m}}\leq C_{m}\norm{u}_{C^{0}}$.
\end{lemma}

\begin{lemma}\label{Lem:E0:eps}
 For $n\in\N$ there is a $C_{n}>0$ such that $\norm{(E_{0}^{c}A(\eps\cdot))}_{C^{n}}=\norm{(E_{0}A(\eps \cdot))-A(\eps\cdot)}_{C^{n}}\leq C_{n}\eps^{n}\norm{A}_{C^{n}}$.
\end{lemma}

\begin{lemma}\label{Lem:Ec:vanish}
 For $u_{1},u_{2}\in C^{n}$ and $r_{1},r_{2}\in\N$ it is true that 
 \begin{equation*}
  E_{c}(\del_{x}^{r_{1}}E_{c}u_{1}\cdot \del_{x}^{r_{2}}E_{c}u_{2})=0.
 \end{equation*}
\end{lemma}

The last statement essentially says that the product of two functions with critical Fourier modes only contains uncritical modes. However, since we have to deal with non-linearities which are in general convolutions, we will need an extension of Lemma~\ref{Lem:Ec:vanish}. In order to proof this, we also require the following well-known result about the convolution of distributions (see for example \cite{DuK10}).

\begin{lemma}\label{Lem:convolution:distribution}
 Let $u,v\in\Sw'$ and assume that either $u$ or $v$ has compact support. Then the convolution $u\ast v$ exists in $\Sw'$ and moreover, it holds $\F(u\ast v)=\F(u)\F(v)$. This means in particular that the product on the right-hand side exists in $\Sw'$.
\end{lemma}

\begin{remark}\label{Rem:product:distributions}
 As a consequence of Lemma~\ref{Lem:convolution:distribution} it also holds that $\F(uv)=(2\pi)^{-1}\F(u)\ast \F(v)$ provided that $u,v\in\Sw'$ such that either $\F(u)$ or $\F(v)$ has compact support.
\end{remark}

We can then show the following generalisation of Lemma~\ref{Lem:Ec:vanish}.

\begin{lemma}\label{Lem:cancel:critical}
 For all $n\in\Z$ and $B_{1}, B_{2}\in\Sw'$ with Fourier transform supported in $\overline{\ball}_{1/4}(-1)\cup \overline{\ball}_{1/4}(1)$ it holds 
 \begin{equation*}
  E_{c}\bigl(B_{1}(Q\ee^{n\im \cdot})\ast B_{2}\bigr)=0.
 \end{equation*}
 In particular $B_{1}(Q\ee^{n\im \cdot})\ast B_{2}\in\Sw'$ is well-defined. 
\end{lemma}

\begin{proof}
 Due to the assumptions on $Q$ it is well-known that $\F(Q\ee^{\im n\cdot})\in C^{\infty}$. Thus, since $\F(B_{1})$ and $\F(B_{2})$ are assumed to have compact support, Lemma~\ref{Lem:convolution:distribution} yields that $B_{1}(Q\ee^{n\im \cdot})\ast B_{2}\in \Sw'$ exists and
 \begin{equation}\label{eq:cancel:crit:1}
  \F\bigl(B_{1}(Q\ee^{n\im \cdot})\ast B_{2}\bigr)=\F(B_{1})\ast \bigl(\F(Q\ee^{n\im \cdot})\F(B_{2})\bigr).
 \end{equation}
 Since $\F(B_{2})$ is supported in $\overline{\ball}_{1/4}(-1)\cup \overline{\ball}_{1/4}(1)$, the same is true for $\F(Q\ee^{n\im \cdot})\F(B_{2})$ as well as for $\F(B_{1})$ by assumption. Thus, we immediately obtain from~\eqref{eq:cancel:crit:1} that the support of $\F\bigl(B_{1}(Q\ee^{n\im \cdot})\ast B_{2}\bigr)$ is contained in 
 $\Omega\vcc=\overline{\ball}_{1/2}(-1)\cup \overline{\ball}_{1/2}(0)\cup \overline{\ball}_{1/2}(1)$ while $\chi_{c}\equiv 0$ on $\Omega$. Thus the claim immediately follows from the definition of $E_{c}$.
\end{proof}

In a similar fashion, we have the following result which provides information on the support in Fourier space for the operators induced by $Q$ and $K$.

\begin{lemma}\label{Lem:support:Fourier}
 For all $n\in\Z$ and $B_{1}, B_{2}, B_{3}\in\Sw'$ with Fourier transform supported in $\overline{\ball}_{1/4}(0)$ the expressions $B_{1}(Q\ee^{n\im \cdot})\ast B_{2}$ and $B_{1}(K\ee^{n\im \cdot})\ast (B_{2}B_{3})$ are well-defined in $\Sw'$ and it holds
 \begin{equation*}
  \supp \F\bigl(B_{1}(Q\ee^{n\im \cdot})\ast B_{2}\bigr)\subset \overline{\ball}_{1/2}(0)\quad \text{and}\quad \supp \F\bigl(B_{1}(K\ee^{n\im \cdot})\ast (B_{2}B_{3})\bigr)\subset \overline{\ball}_{3/4}(0).
 \end{equation*}
\end{lemma}

\begin{proof}
 Similarly as in the proof of Lemma~\ref{Lem:cancel:critical} one finds together with \cref{Lem:convolution:distribution,Rem:product:distributions} that $B_{1}(Q\ee^{n\im \cdot})\ast B_{2}$ and $B_{1}(K\ee^{n\im \cdot})\ast (B_{2}B_{3})$ are well-defined and it holds
 \begin{align*}
  \F\bigl(B_{1}(Q\ee^{n\im \cdot})\ast B_{2}\bigr)&=\F(B_{1})\ast \bigl(\F(Q\ee^{n\im \cdot})\F(B_{2})\bigr)
  \shortintertext{and}
  \F\bigl(B_{1}(K\ee^{n\im \cdot})\ast (B_{2}B_{3})\bigr)&=\F(B_{1})\ast \bigl(\F(K\ee^{\im n\cdot})(\F(B_{2})\ast \F(B_{3})\bigr).
 \end{align*}
From these relations, the claim immediately follows due to the assumptions on the support of $\F(B_{1})$, $\F(B_{2})$ and $\F(B_{3})$.
\end{proof}

For later use, we also recall the following semi-group estimates which are stated in~\cite{Schneider4}.

\begin{lemma}\label{Lem:semi:group}
 Let $\ee^{\Lin t}$ denote the semi-group associated to the operator $\Lin$. Then there exist constants $C,\sigma>0$ which are independent of $\eps$ such that it holds
 \begin{equation*}
  \norm{\ee^{\Lin t}E_{c}^{h}}_{\mathcal{L}(C^{1},C^{4})}\leq C\ee^{\eps^{2} t}\quad \text{and}\quad \norm{\ee^{\Lin t}E_{s}^{h}}_{\mathcal{L}(C^{1},C^{4})}\leq C\ee^{-\sigma t}\max\{1,t^{-3/4}\}.
 \end{equation*}
\end{lemma}

 \section{The residuum}\label{Sec:residuum}
 
 In this section, we will compute the residuum as defined in~\eqref{eq:Res} and moreover, we will derive several estimates which we will need for the proof of the main statement.
 
 \subsection{Computing the residuum}
 
 Since we only need estimates on the $C^{1}$-norm of $\Res(\phi)$ one can easily verify, that the assumptions of Theorem~\ref{Thm:main} together with Lemma~\ref{Lem:smoothing} yield that all derivatives which occur during the computation of $\Res(\phi)$ are uniformly bounded on the relevant time interval. More precisely, this is immediately clear for the purely spatial derivatives. However, the following lemma states that also the $C^{1}$-norm of the time derivative is uniformly bounded. 
 
 \begin{lemma}\label{Lem:est:time:derivative}
 Let $A\in C^{0}([0,T_{*}],C_{\text{b}}^{4}(\R))$ be a solution of~\eqref{eq:GL} and $A_{0}$ and $A_{2}$ be given as in~\eqref{eq:A2:A0}. Then it holds
 \begin{equation*}
  \norm{\del_{T}A_{0}}_{C^{1}}+\norm{\del_{T}A_{2}}_{C^{1}}\leq C\bigl(\norm{A}_{C^{3}}+\norm{A}_{C^{1}}^{3}\bigr)\norm{A}_{C^{1}}
 \end{equation*}
 for some constant $C>0$.
\end{lemma}

\begin{proof}
Due to~\eqref{eq:A2:A0} it holds $A_{0}=-2q_{1}A\bar{A}$ and $A_{2}=-(q_{1}A^2)/9$. Thus, we have
\begin{equation*}
 \del_{T}A_{0}=-2q_{1}(\del_{T}A\bar{A}+A\del_{T}\bar{A})\quad \text{and}\quad \del_{T}A_{2}=-\frac{2}{9}q_{1}A\del_{T}A.
\end{equation*}
Since both $A$ and $\bar{A}$ solve~\eqref{eq:GL} the claim easily follows.
\end{proof}
 
As a consequence, it suffices to consider only terms up to $\Ord(\eps^{3})$ and we will therefore only compute explicitly these terms while all expressions of $\Ord(\eps^{4})$ are just estimated by a constant.

To simplify the presentation, we first compute the different expressions separately and then finally collect all the terms. Moreover, we skip the argument of the functions in order to improve the readability and we use the common notation $\cc$ to indicate complex conjugate. First of all, we obtain
 \begin{equation*}
  \del_{t}\phi=-\eps^{3}\del_{T}(E_{0}A)\ee^{\im x}+\cc+ \Ord(\eps^{4}).
 \end{equation*}
 Moreover, it holds $(1+\del_{x}^{2})^{2}=1+2\del_{x}^{2}+\del_{x}^{4}$ and we have
 \begin{multline*}
  -2\del_{x}^{2}\phi=\Bigl[-2\eps^{3}\del_{X}^{2}(E_{0}A)\ee^{\im x}-4\im \eps^{2}\del_{X}(E_{0}A)\ee^{\im x} +2\eps(E_{0}A)\ee^{\im x}\\*
  +8\eps^{2}(E_{0}A_{2})\ee^{2\im x}-8\im \eps^{3}\del_{X}(E_{0}A_{2})\ee^{2\im x}\Bigr]+\cc+\Ord(\eps^{4}).
 \end{multline*}
Similarly, we obtain
\begin{multline*}
 -\del_{x}^{4}\phi=\Bigl[6\eps^{3}\del_{X}^{2}(E_{0}A)\ee^{\im x}+4\im \eps^{2}\del_{X}(E_{0}A)\ee^{\im x}
 -\eps(E_{0}A)\ee^{\im x}\\*
 -16\eps^{2}(E_{0}A_{2})\ee^{2\im x}+32\im \eps^{3}\del_{X}(E_{0}A_{2})\ee^{2\im x}\Bigr]+\cc+\Ord(\eps^{4}).
\end{multline*}
If we also note that
$\eps^{2}\phi=\eps^{3}(E_{0}A)\ee^{\im x}+\cc+\Ord(\eps^{4})$ we already get
\begin{multline}\label{eq:Res:lambda}
 \Lin(\phi)=\Bigl[-9\eps^{2}(E_{0}A_{2})\ee^{2\im x}+4\eps^{3}\del_{X}^{2}(E_{0}A)\ee^{\im x}+(E_{0}A)\ee^{\im x}+24\im \eps^{3}\del_{X}(E_{0}A_{2})\ee^{-2\im x}\Bigr]+\cc\\*
 -\eps^{2}(E_{0}A_{0})+\Ord(\eps^4).
\end{multline}
In order to compute the non-linear terms, we will use the general relation
\begin{multline*}
 V(\eps x) \bigl(N(\cdot)\ast W(\eps\cdot) \ee^{m\im \cdot}\bigr)(x)=V(\eps x)\ee^{n\im x}\int_{\R}N(y)W(\eps(x-y))\ee^{m\im (x-y)}\dy\\*
 =V(\eps x)\ee^{(n+m)\im x}\int_{\R}N(y)\ee^{-m\im y}W(\eps(x-y))\dy=V(\eps x)\ee^{(n+m)\im x} \bigl((N(\cdot)\ee^{-m\im \cdot})\ast W(\eps \cdot)\bigr)(x).
\end{multline*}
We note that these manipulations are rigorously justified in the expressions where we will use this below. In particular, we find
\begin{equation}\label{eq:Res:quadratic}
 \begin{split}
   \Nlin_{Q}(\phi)=&-\eps^{2}\Bigl[(E_{0}A)\ee^{2\im x}\bigl((Q\ee^{-\im \cdot})\ast(E_{0}A)(\eps\cdot)\bigr)+(E_{0}A)\bigl((Q\ee^{\im \cdot})\ast(E_{0}\bar{A})(\eps \cdot)\bigr)\Bigr]+\cc\\*
 &-\eps^{3}\Bigl[(E_{0}A)\ee^{3\im x}\bigl((Q\ee^{-2\im \cdot})\ast(E_{0}A_{2})(\eps\cdot)\bigr)+(E_{0}A_{2})\ee^{3\im x}\bigl((Q\ee^{-\im \cdot})\ast(E_{0}A)(\eps\cdot)\bigr)\\*
 &\qquad+(E_{0}\bar{A})\ee^{\im x}\bigl((Q\ee^{-2\im \cdot})\ast(E_{0}A_{2})(\eps \cdot)\bigr)+(E_{0}A_{2})\ee^{\im x}\bigl((Q\ee^{\im \cdot})\ast(E_{0}\bar{A})(\eps\cdot)\bigr)\\*
 &\qquad+(E_{0}A)\ee^{\im x}\bigl(Q\ast(E_{0}A_{0})(\eps\cdot)\bigr)+(E_{0}A_{0})\ee^{\im x}\bigl((Q\ee^{-\im \cdot})\ast(E_{0}A)(\eps\cdot)\bigr)\Bigr]+\cc+\Ord(\eps^{4}).
 \end{split}
\end{equation}
 For the cubic terms we obtain in the same way
 \begin{multline}\label{eq:Res:cubic}
  \Nlin_{K}(\phi)=-\eps^{3}\Bigl[(E_{0}A)\ee^{3\im x}\bigl((K\ee^{-2\im \cdot})\ast(E_{0}A)^{2}(\eps \cdot)\bigr)+(E_{0}\bar{A})\ee^{\im x}\bigl((K\ee^{-2\im \cdot})\ast(E_{0}A)^{2}(\eps \cdot)\bigr)\\*
  +2(E_{0}A)\ee^{\im x}\bigl(K\ast \bigl((E_{0}A)(E_{0}\bar{A})\bigr)(\eps\cdot)\bigr)\Bigr]+\cc+\Ord(\eps^{4}).
 \end{multline}
 Summarising~\cref{eq:Res:lambda,eq:Res:quadratic,eq:Res:cubic} we find that 
 \begin{equation*}
  \Res(\phi)=\sum_{\ell=-3}^{3}a_{\ell} \ee^{\im \ell x} +\Ord(\eps^{4})
 \end{equation*}
 with $a_{-\ell}=\bar{a}_{\ell}$ and
 \begin{equation}\label{eq:Res:prefactors}
   \begin{aligned}
   a_{0}&=-\eps^{2}\Bigl[(E_{0}A_{0})+(E_{0}A)\bigl((Q\ee^{\im \cdot})\ast (E_{0}\bar{A})(\eps\cdot)\bigr)+(E_{0}\bar{A})\bigl((Q\ee^{-\im \cdot})\ast(E_{0}A)(\eps \cdot)\bigr)\Bigr]\\
   a_{1}&=\eps^{3}\Bigl[-\del_{T}(E_{0}A)+4\del_{X}^{2}(E_{0}A)+(E_{0}A)-(E_{0}\bar{A})(Q\ee^{-2\im\cdot})\ast(E_{0}A_2)(\eps\cdot)\\
   &\quad\phantom{{}=\eps^{3}\Bigl[}-(E_{0}A)Q\ast (E_{0}A_{0})(\eps\cdot)-(E_{0}A_{0})(Q\ee^{-\im \cdot})\ast (E_{0}A)(\eps\cdot)\\
   &\qquad\phantom{{}=\eps^{3}\Bigl[}-2(E_{0}A)K\ast \bigl((E_{0}A)(E_{0}\bar{A})\bigr)(\eps\cdot)-(E_{0}\bar{A}) (K\ee^{-2\im \cdot})\ast (E_{0}A)^{2}(\eps\cdot)\Bigr]\\
   a_{2}&=-9\eps^{2}(E_{0}A_{2})+24\im \eps^{3}\del_{X}(E_{0}A_{2})-\eps^{2}(E_{0}A)(Q\ee^{-\im\cdot})\ast (E_{0}A)(\eps\cdot)\\
   a_{3}&=-\eps^{3}\Bigl[(E_{0}A)(Q\ee^{-2\im\cdot})\ast (E_{0}A_{2})(\eps\cdot)\\
   &\quad\phantom{{}=-\eps^{3}\Bigl[}+(E_{0}A_{2})(Q\ee^{-\im\cdot})\ast(E_{0}A)(\eps\cdot)+(E_{0}A)(K\ee^{-2\im\cdot})\ast (E_{0}A)^{2}(\eps\cdot)\Bigr].
   \end{aligned}
 \end{equation}
 
 \subsection{Estimating the residuum}
 
  In this section, we provide several estimates on $\Res(\phi)$ that we will need later on. More precisely, the next lemma states that the pre-factor for the uncritical modes is of order $\eps^{3}$ while that one for the critical modes can even be bounded by $\eps^{4}$.

 \begin{lemma}\label{Lem:Res:est:prefactors}
  There exists a constant $C>0$ such that it holds
  \begin{align}
   \sup_{t\in[0,T_{*}/\eps^{2}]}\bigl(\norm{a_{0}}_{C_{1}}+\norm{a_{2}}_{C^{1}}+\norm{a_{3}}_{C^{1}}\bigr)&\leq C\eps^{3} \label{eq:Res:est:prefactor:1}\\
   \sup_{t\in[0,T_{*}/\eps^{2}]}\norm{a_{1}}_{C^{1}}&\leq C\eps^{4} \label{eq:Res:est:prefactor:2}
  \end{align}
 where $a_{\ell}$ is given by~\eqref{eq:Res:prefactors} for $\ell\in\{0,1,2,3\}$
 \end{lemma}
 
 The following relations will be used in the proof of the lemma.
 
 \begin{remark}\label{Rem:relations}
For each $n\in\Z$ and functions $B,B_1,B_2,B_3$ we have the relations $(E_{0}B)=(E_{0}^{c}B)+B$ as well as
\begin{equation*}
 (E_{0}B_{1})(Q\ee^{n\im\cdot})\ast(E_{0}B_{2})=(E_{0}^{c}B_{1})(Q\ee^{n\im\cdot})\ast (E_{0}B_{2})+B_{1}(Q\ee^{n\im\cdot})\ast(E_{0}^{c}B_{2})+B_{1}(Q\ee^{n\im\cdot})\ast B_{2}
\end{equation*}
and
\begin{multline*}
 (E_{0}B_{1})(K\ee^{n\im\cdot})\ast\bigl((E_{0}B_{2})(E_{0}B_{3})\bigr)=(E_{0}^{c}B_{1})(K\ee^{n\im\cdot})\ast\bigl((E_{0}B_{2})(E_{0}B_{3})\bigr)\\*
 +B_{1}(K\ee^{n\im\cdot})\ast \bigl((E_{0}^{c}B_{2})(E_{0}B_{3})\bigr)+B_{1}(K\ee^{n\im\cdot})\ast\bigl(B_{2}(E_{0}^{c}B_{3})\bigr)+B_{1}(K\ee^{n\im\cdot})\ast(B_{2}B_{3}).
\end{multline*}
\end{remark}
 
 \begin{proof}[Proof of Lemma~\ref{Lem:Res:est:prefactors}]
  We consider first $a_{0}$. Since $A_{0}=-2q_{1}A\bar{A}$ we obtain by means of Remark~\ref{Rem:relations} that
  \begin{multline*}
   -\eps^{2}\Bigl[(E_{0}A_{0})+(E_{0}A)\bigl((Q\ee^{\im \cdot})\ast (E\bar{A})(\eps\cdot)\bigr)+(E_{0}\bar{A})\bigl((Q\ee^{-\im \cdot})\ast(E_{0}A)(\eps \cdot)\bigr)\Bigr]\\*
   =-\eps^{2}\Bigl[(E_{0}^{c}A_{0})-2q_{1}A\bar{A}+(E_{0}^{c}A)\bigl((Q\ee^{\im \cdot})\ast (E_{0}\bar{A})(\eps\cdot)\bigr)+A\bigl((Q\ee^{\im \cdot})\ast (E_{0}^{c}\bar{A})(\eps\cdot)\bigr)+A(Q\ee^{\im \cdot})\ast\bar{A}(\eps\cdot)\\*
   +(E_{0}^{c}\bar{A})\bigl((Q\ee^{-\im \cdot})\ast(E_{0}A)(\eps \cdot)\bigr)+\bar{A}\bigl((Q\ee^{-\im \cdot})\ast(E_{0}^{c}A)(\eps \cdot)\bigr)+\bar{A}\bigl((Q\ee^{-\im \cdot})\ast A(\eps \cdot)\Bigr].
  \end{multline*}
  Due to \cref{Lem:approx,Lem:continuity:convolution} and $A_{0}=-2q_{1}A\bar{A}$ we thus obtain
  \begin{multline*}
   \phantom{{}\leq{}}\norm{(E_{0}A_{0})+(E_{0}A)\bigl((Q\ee^{\im \cdot})\ast (E\bar{A})(\eps\cdot)\bigr)+(E_{0}\bar{A})\bigl((Q\ee^{-\im \cdot})\ast(E_{0}A)(\eps \cdot)\bigr)}_{C^{1}}\\*
   \shoveleft{\leq C\eps \norm{A}_{C^{1}}\norm{\bar{A}}_{C^{1}}}\\*
   +C\Bigl[\norm{E_{0}^{c}A}_{C^{1}}\norm{E_{0}\bar{A}}_{C^{1}}+\norm{A}_{C^{1}}\norm{E_{0}^{c}\bar{A}}_{C^{1}}+\norm{E_{0}^{c}\bar{A}}_{C^{1}}\norm{E_{0}A}_{C^{1}}+\norm{\bar{A}}_{C^{1}}\norm{E_{0}^{c}A}_{C^{1}}\Bigr].
  \end{multline*}
 \Cref{Lem:E0:eps,Lem:smoothing} together with the uniform boundedness of $A$ thus yield
 \begin{equation*}
  \sup_{t\in[0,T_{*}/\eps^{2}]}\norm{a_{0}}_{C^{1}}\leq C\eps^{3}.
 \end{equation*}
 To estimate $a_{2}$ we can proceed in the same way, i.e.\@ Lemma~\ref{Lem:smoothing} together with the boundedness of $A$ yields $\norm{24\im \eps^{3}\del_{X}(E_{0}A_{2})}_{C^{1}}\leq C\eps^{3}$. Thus, it remains to estimate $-9(E_{0}A_{2})-(E_{0}A)(Q\ee^{-\im\cdot})\ast (E_{0}A)(\eps\cdot)$ which can be rewritten by means of \cref{Rem:relations,eq:A2:A0} as
 \begin{multline*}
  -9(E_{0}A_{2})-(E_{0}A)(Q\ee^{-\im\cdot})\ast (E_{0}A)(\eps\cdot)\\*
  =-9(E_{0}^{c}A_{2})+q_{1}A^{2}-(E_{0}^{c}A)(Q\ee^{-\im \cdot})\ast(E_{0}A)(\eps\cdot)-A(Q\ee^{-\im\cdot})\ast (E_{0}^{c}A)-A(Q\ee^{-\im\cdot})\ast A(\eps\cdot).
 \end{multline*}
 \Cref{Lem:approx,Lem:continuity:convolution,Lem:E0:eps,Lem:smoothing} as well as $A_{2}=-(q_1 A^2)/9$ and the uniform boundedness of $A$ then imply that
 \begin{equation*}
  \norm{-9(E_{0}A_{2})-(E_{0}A)(Q\ee^{-\im\cdot})\ast (E_{0}A)(\eps\cdot)}_{C^{1}}\leq C\eps
 \end{equation*}
 uniformly with respect to $t\in[0,T_{*}/\eps^{2}]$. 
 
Moreover, due to the choice of $A_{2}$ together with \cref{Lem:continuity:convolution,Lem:smoothing} one immediately gets $\norm{a_{3}}_{C^{1}}\leq C\eps^{3}$ for all $t\in[0,T_{*}/\eps^{2}]$. Summarising, this shows~\eqref{eq:Res:est:prefactor:1}.

Thus, it only remains to prove~\eqref{eq:Res:est:prefactor:2} and for this we proceed similarly as before. More precisely, we first note that Remark~\ref{Rem:relations} allows to rewrite
\begin{equation*}
 -\del_{T}(E_{0}A)+4\del_{X}^{2}(E_{0}A)+(E_{0}A)=-\del_{T}(E_{0}^{c}A)+4\del_{X}^{2}(E_{0}^{c}A)+(E_{0}^{c}A)-\del_{T}A+4\del_{X}^{2}A+A.
\end{equation*}
Since $A$ solves~\eqref{eq:GL} we further get
\begin{multline*}
 -\del_{T}(E_{0}A)+4\del_{X}^{2}(E_{0}A)+(E_{0}A)=-\del_{T}(E_{0}^{c}A)+4\del_{X}^{2}(E_{0}^{c}A)+(E_{0}^{c}A)\\*
 +\Bigl(2k_{0}+k_{2}-\frac{q_{1}q_{2}}{9}-\frac{q_{1}^{2}}{9}-2q_{0}q_{1}-2q_{1}^{2}\Bigr)\abs*{A}^{2}A.
\end{multline*}
Therefore, it remains to estimate the $C^{1}$-norm of
\begin{multline*}
 -\del_{T}(E_{0}^{c}A)+4\del_{X}^{2}(E_{0}^{c}A)+(E_{0}^{c}A)+\Bigl(2k_{0}+k_{2}-\frac{q_{1}q_{2}}{9}-\frac{q_{1}^{2}}{9}-2q_{0}q_{1}-2q_{1}^{2}\Bigr)\abs*{A}^{2}A\\*
 +(E_{0}A)-(E_{0}\bar{A})(Q\ee^{-2\im\cdot})\ast(E_{0}A_2)(\eps\cdot)-(E_{0}A)Q\ast (E_{0}A_{0})(\eps\cdot)
 -(E_{0}A_{0})(Q\ee^{-\im \cdot})\ast (E_{0}A)(\eps\cdot)\\*
 -2(E_{0}A)K\ast \bigl((E_{0}A)(E_{0}\bar{A})\bigr)(\eps\cdot)-(E_{0}\bar{A}) (K\ee^{-2\im \cdot})\ast (E_{0}A)^{2}(\eps\cdot).
\end{multline*}
However, since $\abs{A}^{2}A=A^{2}\bar{A}$ this can be done in the same way as for $a_{0}$ and $a_{2}$.
 \end{proof}
 
 As a consequence of Lemma~\ref{Lem:Res:est:prefactors}, we can now prove the following result which provides bounds on the restrictions of $\Res(\phi)$ to critical and uncritical Fourier modes.
 
 \begin{proposition}\label{Prop:residuum:estimates}
  For each solution $A\in C([0,T_{*}],C^{4}_{\text{b}})$ of~\eqref{eq:GL} and $\phi$ as in~\eqref{eq:approx:refined} there exists a constant such that it holds
  \begin{equation*}
   \sup_{t\in[0,T_{*}/\eps^{2}]}\norm{E_{s}(\Res(\phi))}_{C^{1}}\leq C\eps^{3}\quad \text{and}\quad \sup_{t\in[0,T_{*}/\eps^{2}]}\norm{E_{c}(\Res(\phi))}_{C^{1}}\leq C\eps^{4}.
  \end{equation*}
 \end{proposition}

 \begin{proof}
  The proof follows easily from Lemma~\ref{Lem:Res:est:prefactors}. Precisely, we note that $\Res(\phi)=\sum_{\ell=-3}^{3}a_{\ell}\ee^{\im \ell x}$ with $a_{\ell}$ as in~\eqref{eq:Res:prefactors} and $a_{-\ell}=\bar{a}_{\ell}$. Moreover, $E_{s}=1-E_{c}$ and thus, due to Lemma~\ref{Lem:smoothing} we deduce that $E_{s}\colon C^{1}\to C^{1}$ is linear and bounded. Therefore, in order to verify the first claimed estimate, it suffices to show that
  \begin{equation*}
   \sup_{t\in[0,T_{*}/\eps^{2}]}\norm{\Res(\phi)}_{C^{1}}\leq C\eps^{3}
  \end{equation*}
 which is however an immediate consequence of Lemma~\ref{Lem:Res:est:prefactors}.
 
 To prove the second claim of the lemma, we note that the definition of $E_{0}$ together with Lemma~\ref{Lem:support:Fourier} yields that $a_{\ell}$ is supported in $[-3/4,3/4]$. Thus, we find that the Fourier transform of $a_{\ell}\ee^{\im\ell \cdot}$ is supported in $\overline{B}_{3/4}(\ell)$. Since $\chi_{c}\equiv 0$ on $\bar{B}_{3/4}(\ell)$ for $\ell\in\{0,\pm 2,\pm 3\}$ we get that $E_{c}(\Res(\phi))=E_{c}(a_{1}\ee^{\im x}+\bar{a}_{1}\ee^{-\im x})+\Ord(\eps^{4})$. However, Lemma~\ref{Lem:smoothing} implies that $E_{c}\colon C^{1}\to C^{1}$ is bounded and thus the second claim of the lemma follows immediately from Lemma~\ref{Lem:Res:est:prefactors}.
 \end{proof}

\section{An equation for the approximation error}\label{Sec:equation:error}
 
 In this section, we will derive the equation which the approximation error $R$ has to satisfy and we will mainly use the same notation as in~\cite{Schneider4}. As already mentioned before, it will be necessary to treat the critical Fourier modes $\ee^{\pm \im x}$ separately from the uncritical ones and we therefore write
 \begin{equation*}
  u=\eps\phi_{c}+\eps^{2}\phi_{s}+\eps^{2}R_{c}+\eps^{3}R_{s}
 \end{equation*}
 where $\phi_{c}$ and $\phi_{s}$ have been defined in~\eqref{eq:phi:c:phi:s}. Moreover, to shorten the notation we also use $R\vcc=\eps^2 R_{c}+\eps^{3}R_{s}$ such that it holds $u=\phi+R$. If we plug this into~\eqref{eq:SW} it follows
 \begin{equation*}
  \begin{split}
   0&=\del_{t}u-\Lin(u)-\Nlin(u)\\
   &=\del_{t}R+\del_{t}\phi-\Lin(R)-\Lin(\phi)-\Nlin_{Q}(R+\phi)-\Nlin_{K}(R+\phi)\\
   &=\del_{t}R+\del_{t}\phi-\Lin(R)-\Lin(\phi)-\Nlin_{Q}(R)-\Nlin_{Q}(\phi)+R(Q\ast\phi)+\phi(Q\ast R)\\
   &\phantom{{}={}} -\Nlin_{K}(R)-\Nlin_{K}(\phi)+2R K\ast (R\phi)+RK\ast \phi^2+2\phi K\ast (R\phi)+\phi K\ast R^2.
  \end{split}
 \end{equation*}
If we now insert $R=\eps^2 R_{c}+\eps^{3}R_{s}$ and recall that $\Res(\phi)=-\del_{t}\phi+\Lin(\phi)+\Nlin(\phi)$ this can be further rearranged as
\begin{multline*}
 \eps^{2}\del_{t}R_{c}+\eps^{3}\del_{t}R_{s}=\eps^2\Lin(R_{c})+\eps^{3}\Lin(R_{s})+\Res(\phi) -\eps^{4}R_{c}Q\ast R_{c}-\eps^{5}R_{s}Q\ast (R_{c}+\eps R_{s})\\*
 -\eps^{6}(R_{c}+\eps R_{s})\bigl(K\ast (R_{c}+\eps R_{s})^{2}\bigr)-\eps^{3}R_{c}Q\ast \phi_{c}-\eps^{4}R_{c} Q\ast \phi_{s}\\*
 -\eps^{4}R_{s}Q\ast \phi_{c}-\eps^{5}R_{s}Q\ast \phi_{s}-\eps^{3}\phi_{c}Q\ast R_{c}-\eps^{4}\phi_{s}Q\ast R_{c}\\*
 -\eps^{4}\phi_{c}Q\ast R_{s}-\eps^{5}\phi_{s}Q\ast R_{s}-2\eps^{5}(R_{c}+\eps R_{s})\bigl(K\ast \bigl((R_{c}+\eps R_{s})(\phi_{c}+\eps \phi_{s})\bigr)\bigr)\\*
 -\eps^{4}R_{c}(K\ast \phi_{c}^{2})-\eps^{5}R_{c}\bigl(K\ast (2\phi_{c}\phi_{s}+\eps\phi_{s}^{2})\bigr)-\eps^{5}R_{s}\bigl(K\ast (\phi_{c}+\eps\phi_{s})^{2}\bigr)-2\eps^{4}\phi_{c}(K\ast (R_{c}\phi_{c})\bigr)\\*
 -2\eps^{5}\phi_{c}\bigl(K\ast(R_{c}\phi_{s}+R_{s}\phi_{c}+\eps R_{s}\phi_{s})\bigr)-2\eps^{5}\phi_{s}\bigl(K\ast \bigl((R_{c}+\eps R_{s})(\phi_{c}+\eps\phi_{s})\bigr)\bigr)\\*
 -2\eps^{5}(\phi_{c}+\eps\phi_{s})\bigl(K\ast (R_{c}+\eps R_{s})^{2}\bigr).
\end{multline*}
If we divide by $\eps^{2}$ and reorganise, we finally end up with
\begin{multline}\label{eq:approximation:error}
 \del_{t}R_{c}+\eps \del_{t}R_{s}=\Lin(R_{c})+\eps \Lin(R_{s})-\eps L_{2}(R_{c})-\eps ^{2}N_{2}(R_{c})\\*
 -\eps^{2}L_{1}(R_{c},R_{s})+\eps^{3}N_{1}(R_{c},R_{s},\eps)+\frac{1}{\eps^{2}}\Res(\phi),
\end{multline}
where we write
\begin{equation*}
 \begin{split}
  L_{2}(R_{c})&=R_{c}Q\ast \phi_{c}+\phi_{c} Q\ast R_{c}\\
  N_{2}(R_{c})&=R_{c}Q\ast R_{c}\\
  L_{1}(R_{c},R_{s})&=R_{c}Q\ast \phi_{s}+R_{s}Q\ast \phi_{c}+\phi_{s}Q\ast R_{c}+\phi_{c}Q\ast R_{s}+R_{c}K\ast \phi_{c}^{2}+2\phi_{c}\bigl(K\ast(R_{c}\phi_{c})\bigr)
 \end{split}
\end{equation*}
and
\begin{multline*}
 N_{1}(R_{c},R_{s},\eps)=-R_{s}Q\ast (R_{c}+\eps R_{s})-\eps (R_{c}+\eps R_{s})\bigl(K\ast (R_{c}+\eps R_{s})^2\bigr)-R_{s}Q\ast \phi_{s}\\*
  -\phi_{s}Q\ast R_{s}-2(R_{c}+\eps R_{s})\bigl(K\ast \bigl(R_{c}+\eps R_{s})(\phi_{c}+\eps \phi_{s})\bigr)\bigr)-R_{c}\bigl(K\ast (2\phi_{c}\phi_{s}+\eps \phi_{s}^{2})\bigr)\\*
  -R_{s}\bigl(K\ast (\phi_{c}+\eps \phi_{s})^{2}\bigr)-2\phi_{c}\bigl(K\ast(R_{c}\phi_{s}+R_{s}\phi_{c}+\eps R_{s}\phi_{s})\bigr)\\*
  -2\phi_{s}\bigl(K\ast \bigl((R_{c}+\eps R_{s})(\phi_{c}+\eps\phi_{s})\bigr)\bigr)-2(\phi_{c}+\eps\phi_{s})\bigl(K\ast (R_{c}+\eps R_{s})^{2}\bigr).
\end{multline*}

As in~\cite{Schneider4} we now exploit that Lemma~\ref{Lem:cancel:critical} implies $E_{c}L_{2}(R_{c})=0$ and $E_{c}N_{2}(R_{c})=0$ to separate the equation for $R$. Precisely, we apply the identity operator $\text{Id}=E_{c}+E_{s}$ to~\eqref{eq:approximation:error} such that we obtain
\begin{equation}\label{eq:Rc:Rs}
 \begin{split}
 \del_{t}R_{c}&=\Lin(R_{c})-\eps^{2}L_{c}(R_{c},R_{s})+\eps^{3}N_{c}(R_{c},R_{s})+\eps^{2}\delta_{c}\\
 \del_{t}R_{s}&=\Lin(R_{s})-L_{s}(R_{c})+\eps N_{s}(R_{c},R_{s})+\delta_{s},
 \end{split}
 \end{equation}
with the abbreviations
\begin{equation}\label{eq:operators}
 \begin{aligned}
  L_{c}(R)&=E_{c}\bigr(L_{1}(R_{c},R_{s})\bigr) & L_{s}(R_{c})&=E_{s}\bigl(L_{2}(R_{c})\bigr)\\
  N_{c}(R)&=E_{c}\bigl(N_{1}(R_{c},R_{s})\bigr) & N_{s}(R_{c},R_{s})&=E_{s}\bigl(L_{1}(R_{c},R_{s})+N_{2}(R_{c})+\eps N_{1}(R_{c},R_{s})\bigr)\\
  \delta_{c}&=\eps^{-4}E_{c}\bigl(\Res(\phi)\bigr) & \delta_{s}&=\eps^{-3}E_{s}\bigl(\Res(\phi)\bigr). 
 \end{aligned}
\end{equation}

\begin{remark}
 Note that if $R_{c}$ and $R_{s}$ solve~\eqref{eq:Rc:Rs} the sum $R_{c}+\eps R_{s}$ gives a solution to~\eqref{eq:approximation:error}.
\end{remark}

\begin{remark}
 The existence of a unique solution to~\eqref{eq:Rc:Rs} locally in time can be shown by a standard fixed-point argument similarly as in~\cite{Schneider4}. Note that for this it is important that the non-linear terms are locally Lipschitz continuous which might be easily deduced from Lemma~\ref{Lem:continuity:convolution}.
\end{remark}

Moreover, we have the following estimates on the linear operators $L_{c}$ and $L_{s}$.

\begin{lemma}\label{Lem:Lc:Ls}
 There exists a constant $C>0$ such that it holds
 \begin{equation*}
  \norm{L_{c}(R_{c},R_{s})}_{C^{1}}\leq C\bigl(\norm{R_{c}}_{C^{1}}+\norm{R_{s}}_{C^{1}}\bigr)\quad \text{and}\quad \norm{L_{s}(R_{c})}_{C^{1}}\leq C\norm{R_{c}}_{C^{1}}
 \end{equation*}
 for the operators $L_{c}$ and $L_{s}$ as given in~\eqref{eq:operators}.
\end{lemma}

\begin{proof}
 These estimates follow immediately from Lemma~\ref{Lem:continuity:convolution} together with the boundedness of the operators $E_{c}$ and $E_{s}$.
\end{proof}

\section{Proof of Theorem~\ref{Thm:main}}\label{Sec:Proof:main}

Based on the preparations in \cref{Sec:residuum,Sec:equation:error} we will now give the proof of our main result.

\begin{proof}[Proof of Theorem~\ref{Thm:main}]
 We first introduce some notation, namely for fixed $T\geq 0$ and $n\in\N$ we define the Banach space
 \begin{equation*}
  \B_{T}^{n}\vcc=C([0,T],C^{n}(\R)) \quad \text{with norm }\norm{f}_{\B_{T}^{n}}=\sup_{t\in[0,T]}\norm{f(t)}_{C^{n}}.
 \end{equation*}
Moreover, we note that one may easily deduce from Lemma~\ref{Lem:continuity:convolution} together with the boundedness of $E_{c}$ and $E_{s}$ that for each $D>0$ there exists $M_{D}>0$ such that it holds for all $t>0$ that
\begin{equation}\label{eq:apriori:assumption}
 \norm{N_{c}(R_{c},R_{s})}_{\B_{t}^{1}}+\eps \norm{N_{s}(R_{c},R_{s})}_{\B_{t}^{1}}\leq M_{D}\quad \text{if }\norm{R_{c}}_{\B_{t}^{4}}+\eps \norm{R_{s}}_{\B_{t}^{4}}\leq D.
\end{equation}
Furthermore, we recall from Proposition~\ref{Prop:residuum:estimates} that
\begin{equation}\label{eq:bd:delta}
 \norm{\delta_{c}}_{\B_{T_{*}/\eps^{2}}}\leq C\quad \text{and}\quad \norm{\delta_{s}}_{\B_{T_{*}/\eps^{2}}}\leq C.
\end{equation}
Finally, due to the assumptions on the initial data we have
\begin{equation}\label{eq:bd:initial:condition}
 \norm{R_{c}(0)}_{C^{4}}=\norm{R_{c}}_{\B_{0}^{4}}\leq C\quad \text{and}\quad \norm{R_{s}(0)}_{C^{4}}=\norm{R_{s}}_{\B_{0}^{4}}\leq C/\eps.
\end{equation}
By means of the semi-group $\ee^{\Lin t}$ and the relations $E_{c}^{h}E_{c}=E_{c}$ as well as $E_{s}^{h}E_{s}=E_{s}$ we can rewrite~\eqref{eq:Rc:Rs} as
\begin{equation*}
 \begin{split}
  R_{c}(t)&=R_{c}(0)+\eps^{2}\int_{0}^{t}\ee^{\Lin (t-\tau)}E_{c}^{h}\Bigl[-L_{c}(R_{c},R_{s})+\eps N_{c}(R_{c},R_{s})+\delta_{c}\Bigr]\dd{\tau}\\
  R_{s}(t)&=R_{s}(0)+\int_{0}^{t}\ee^{\Lin(t-\tau)}E_{s}^{h}\Bigl[-L_{s}(R_{c})+\eps N_{s}(R_{c},R_{s})+\delta_{s}\Bigr]\dd{\tau}.
 \end{split}
\end{equation*}
From \cref{Lem:semi:group,Lem:Lc:Ls,eq:bd:delta} we thus obtain that
\begin{equation}\label{eq:Rs:bd:1}
  \norm{R_{s}}_{\B_{t}^{4}}\leq \norm{R_{s}}_{\B_{0}^{4}}+C\int_{0}^{t}\max\{1,\tau^{-3/4}\}\ee^{-\sigma\tau}\dd{\tau}\bigl[\norm{R_{c}}_{\B_{t}^{4}}+M_{D}+C\bigr]
\end{equation}
as long as the condition in~\eqref{eq:apriori:assumption} holds. For $R_{c}$ we proceed similarly, while we additionally exploit~\cref{eq:Rs:bd:1,eq:bd:initial:condition} and the assumption $t\leq T_{*}/\eps^2$ to find 
\begin{equation}\label{eq:Rc:bd:1}
 \begin{split}
  \norm{R_{c}}_{\B_{t}^{4}}&\leq \norm{R_{c}}_{\B_{0}^{4}}+C\eps^{2}\int_{0}^{t}\ee^{C\eps^{2}(t-\tau)}\bigl(\norm{R_{c}}_{\B_{\tau}^{4}}+\norm{R_{s}}_{\B_{\tau}^{4}}+\eps M_{D}+C\bigr)\dd{\tau}\\
  &\leq C +C_{T_{*}}\eps^{2}\int_{0}^{t}\norm{R_{c}}_{\B_{\tau}^{4}}\dd{\tau}+C_{T_{*}}(\eps M_{D}+C).
 \end{split}
\end{equation}
Due to Gronwall's inequality and the assumption $t\leq T_{*}/\eps^2$ we obtain
\begin{equation}\label{eq:Rc:bd:2}
 \norm{R_{c}}_{\B_{t}^{4}}\leq C_{T_{*}}\bigl(\eps M_{D}+1\bigr)\ee^{C_{T_{*}}\eps^{2}t}\leq C_{T_{*}}\bigl(\eps M_{D}+1\bigr).
\end{equation}
If we use this estimate together with~\eqref{eq:bd:initial:condition} it follows from~\eqref{eq:Rs:bd:1} that
\begin{equation}\label{eq:Rs:bd:2}
 \norm{R_{s}}_{\B_{t}^{4}}\leq C/\eps+C M_{D}+C_{T_{*}}(\eps M_{D}+1).
\end{equation}
If we now fix first $D_{*}>0$ sufficiently large and then $\eps_{*}=\eps_{*}(D_{*})>0$ sufficiently small one immediately deduces from~\cref{eq:Rc:bd:2,eq:Rs:bd:2} that it holds $\norm{R_{c}}_{\B_{t}^{4}}+\eps \norm{R_{s}}_{\B_{t}^{4}}\leq D_{*}$ for all $t\in[0,T_{*}/\eps^{2}]$ provided $\eps\leq \eps_{*}$. Thus, the error $R=\eps^{2}R_{c}+\eps^{3}R_{s}$ remains in the ball of radius $D_{*}$ (with respect to $\norm{\cdot}_{C^{4}}$) for all $t\in[0,T_{*}/\eps^{2}]$.

Since $R=u-\phi$ we thus find together with Lemma~\ref{Lem:phi:estimates} that
\begin{equation*}
 \norm{u-\psi}_{C^{4}}\leq \norm{R}_{C^{4}}+\norm{\phi-\psi}_{C^{4}}\leq C\eps^{2}.
\end{equation*}
The existence and uniqueness of $u$ now follows straightforward. Precisely, by a standard fixed-point argument one gets that there exists a unique solution $u$ to~\eqref{eq:SW} on a small time interval. Due to the approximation result that we have just shown, this solution cannot blow up---and can thus be extended uniquely---on the interval $[0,T_{*}/\eps^2]$.
\end{proof}

\textbf{Acknowledgments:} CK and ST have been supported by a Lichtenberg 
Professorship of the VolkswagenStiftung. CK also thanks 
Alexander Mielke for very insightful discussions on the development of 
the theory of amplitude equations.

\end{document}